\documentclass{amsart}
\usepackage[latin1]{inputenc}
\usepackage{amssymb,amsfonts,amsmath,amsthm,amscd, epsfig}
\usepackage[all]{xy}
\usepackage{graphicx,color}
\usepackage[english]{babel}

\newtheorem{maintheorem}{Theorem}

\newtheorem{teo}{Theorem}[section]
\newtheorem{defin}[teo]{Definition}
\newtheorem{lema}[teo]{Lemma}
\newtheorem{cor}[teo]{Corollary}
\newtheorem{prop}[teo]{Proposition}
\newtheorem{ex}[teo]{Example}
\newtheorem{obs}[teo]{Remark}

\newcommand{\ku}{\Bbbk}

\newcommand{\vu}{\vspace{.1cm}}

\newcommand{\g}{\mathcal{G}}
\newcommand{\K}{\mathcal{K}}

\begin{document}

\thispagestyle{empty}

\title[Groupoid Actions on Sets, Duality and a Morita Context]{Groupoid Actions on Sets, Duality and a Morita Context}
\author[Della Flora, Fl\^ores, Morgado, Tamusiunas]
{Saradia Della Flora, Daiana Fl\^ores, Andrea Morgado, Tha\'isa Tamusiunas}

\address{Departamento de Matem\'atica, Universidade Federal de Santa Maria,
	97105-900, Santa Maria, RS, Brazil} \email{ saradia.flora@ufsm.br, flores@ufsm.br}

\address{Instituto de F\'isica e Matem\'atica, Departamento de Matem\'atica e Estat\'istica,  Universidade Federal de Pelotas, Caixa Postal 354, 96160-900, Pelotas, RS, Brazil} \email{andrea.morgado@ufpel.edu.br}

\address{Instituto de Matem\'atica, Universidade Federal do Rio Grande do Sul,
	91509-900, Porto Alegre, RS, Brazil} \email{thaisa.tamusiunas@gmail.com}


\maketitle

\begin{abstract}
Let $\g$ and $\K$ be groupoids. We present the notion of a $(\g_{\alpha},\K_{\beta})$-set and we prove a duality theorem in this context, which extends the duality theorem for graded algebras by groups obtained in \cite{a1}. For $A$ a unital $\g$-graded algebra and $X$ a finite split $\g$-set, we show that there is an isomorphism between the category of the left $A$-modules $X$-graded and the category of the left $A \#_{\alpha}^{\g}X$-modules. As an application of this isomorphism, we construct a Morita context. 
\end{abstract}

 \vspace{0,5 cm}



\section{Introduction}

A groupoid is a small category whose morphisms are invertible. It is clear that any group is a groupoid, viewed as a category with a unique object. Thus, it is interessing to investigate under which conditions classic results in the context of groups can be extended to groupoids.

One of the main goal of this paper is to relate the concepts of groupoids grading and groupoids actions on a set. Specifically, we present a generalization for the duality theorem given in \cite{a1} for action and grading of groups. The concept of groupoid grading was introduced in \cite{Lun2} as a natural generalization of group grading. The notion of groupoid action on a set was introduced in \cite{PaqTam}. More precisely, an action of a groupoid $\g$ on a set $X$ is a collection of subsets $X_g$ of $X$ and bijections $\alpha_g: X_{g^{-1}} \to X_g$, $g \in \g$, satisfying some appropriate conditions of compatibility. In this case $X$ is said a $\g$-set. 

Given $A$ a $\g$-graded algebra and $X$ a split $\g$-set, we define a new algebra, called smash product and denoted by $A \#_{\alpha}^{\g}X$; see Section 3. Now let $\K$ be a groupoid and consider the skew groupoid ring $(A\#_{\alpha}^{\g} X)*_{\gamma} \K$, which is well defined according to a compatibility relation between the actions of the groupoids $\g$ and $\K$ on $X$. This relation allows us to define an action of the groupoid $\K$ on the algebra $A \#_{\alpha}^{\g}X$; see Section 4. In this case, $X$ is said $(\g_{\alpha}, \K_{\beta})$-set. With this notation, we have a duality theorem. 

\begin{maintheorem} Let $\g$ and $\K$ be groupoids such that the sets of objects $\g_0$ and $\K_0$ are finite, $A$ be an algebra $\g$-graded and $X$ be a finite split $(\g_{\alpha}, \K_{\beta})$-set. If the action of $\K$ on $X$ is fully faithful, then 
	\begin{align*}
	(A \#_{\alpha}^{\g} X)*_{\gamma} \K \simeq End({A \#_{\alpha}^{\g} X}_{A \#_{\lambda}^{\g} O^{\K}}),
	\end{align*}
	where  $O^{\K}$ is the $\K$-orbit of $X$.
\end{maintheorem}

Another aim of this paper is to describe the category of the left $A \#_{\alpha}^{\g}X$-modules, $_{A \#_{\alpha}^{\g} X}\mathcal{M}$. We show that this category is isomorphic to the category of the left $A$-modules $X$-graded; see Section 5. For each $x \in X$, consider $\g_x$ the $x$-stabilizer in $\g$ and  $A^{\g_x}=\oplus_{g \in \g_x} A_g$. As a consequence of this category isomorphism, we present a condition for that $A^{\g_x}$ and $A \#_{\alpha}^{\g}X$ be Morita equivalent.

\begin{maintheorem} \begin{itemize}
\item[(a)] If $A^{\g_x}$ and $A \#_{\alpha}^{\g}X$ are Morita equivalent and $X_e \neq \emptyset$, for all $e \in \g_0$, then
$\sum\limits_{{g \in \g}\atop{\alpha_g(y)=x}}A_{g^{-1}}A_g=\sum_{g \in \g}A_{d(g)}, \mbox { for all } y \in X.$
\item[(b)] If $\sum\limits_{{g \in \g}\atop{\alpha_g(y)=x}}A_{g^{-1}}A_g=\sum_{g \in \g}A_{d(g)}, \mbox { for all } y \in X$, then $A^{\g_x}$ and $A \#_{\alpha}^{\g}X$ are Morita equivalent.
\end{itemize}
\end{maintheorem}

The paper is organized as follows: in Section 2 we exhibit some preliminary results concerning groupoids, groupoid gradings, groupoid actions on algebras and Galois extension. In Section 3 we present the definitions of $\g$-set and of the smash product $A \#_{\alpha}^{\g}X$. In Section 4 we introduce the definition of $(\g_{\alpha}, \K_{\beta})$-set and we construct an action of the groupoid $\K$ on $A \#_{\alpha}^{\g}X$. Then we prove a duality theorem that extends the duality theorem for graded algebras by groups obtained in \cite{a1}. The Section 5 is dedicated to Morita theory.

Throughout this paper, $\ku$ denotes a field and all algebras are over $\ku$, associative and unital, unless otherwise specified. If $A$ is an algebra, then we denote by $_{A}\mathcal{M}$ the category of left $A$-modules and by $\mathcal{M}_{A}$ the category of right $A$-modules.

\section{Preliminaries}

\subsection{Groupoids and groupoids gradings} \label{groupoid}

A {\it{groupoid}} $\g$ is a small category in which every morphism is invertible. Given $g\in \g$, the {\it domain} and the {\it  range }  of $g$ will be denoted by $d(g)$ and $r(g)$, respectively. Also, $\g_0$ will denote the set of the objects of $\g$. Hence, we have maps $d,r:\g \to \g_0$.  Given $e\in \g_0$, $id_e$ will denote the identity morphism of $e$.  Observe that $id:\g_0\to \g$, given by $id(e)=id_e$, is an injective map. Thus, we identify $\g_0\subseteq \g$. 

Let $\mathcal{G}_{d}\times \!_{r}\mathcal{G}=\{ (g,h) \in \g \times \g : d(g)=r(h)\}$. The map $m: \mathcal{G}_{d}\times \!_{r}\mathcal{G} \to \mathcal{G}$,  $m(g,h)=gh$, is called \emph{composition map}. For each pair of objects $e,f\in \mathcal{G}_0$, denote by $\mathcal{G}(e,f)=\{g \in \g: d(g)=e, r(g)=f\}$. In particular $\mathcal{G}(e,e)=\mathcal{G}_e$ is a group, called {\it isotropy (or principal) group associated to $e$}. For all $e \in \g_{0}$, denote by $D_{e}=\{g\in \mathcal{G}: d(g)=e\}$ and $R_e=\{g\in \mathcal{G}: r(g)=e\}$. Two groupoids $\g$ and $\g'$ are {\it equivalent}, and we write $\g \simeq \g'$, if they are equivalent as categories.


A non-empty subset $\mathcal{H}$ of $\g$ is called a {\it subgroupoid} of $\g$ if is stable by multiplication and by inverse. If, in addition, $\mathcal{H}_{0} = \g_{0}$, then $\mathcal{H}$ is said to be a {\it wide subgroupoid}. 

Consider $\mathcal{H}$ a wide subgroupoid of $\g$. Define the following equivalence relation on $\g$: \begin{center}$g \equiv_{\mathcal{H}}  h$ if and only if $gh^{-1} \in \mathcal{H}$, for all $g, h \in D_e$, $e \in \g_0$. \end{center} We denote by $\g / \mathcal{H}$ the set of right cosets of $\g$ by the relation $\equiv_{\mathcal{H}}$. 

Following  \cite{Lun2}, an algebra $A$ is {\it{graded by a groupoid $\g$}}, or {\it{$\g$-graded}}, if there exists a family $\{A_g\}_{g\in \mathcal{G}}$ of subspaces of $A$ such that $A=\oplus_{g \in \g}A_g$, satisfying
\begin{align*}
A_gA_h
& \left\{\begin{array}{rl}
\subseteq A_{gh}, &\mbox{if } (g,h) \in \mathcal{G}_{d}\times \!_{r}\mathcal{G}\\
=0, & \mbox{otherwise},
\end{array}
\right.
\end{align*}
for all $g,h \in \g$. Clearly, $A_e$ is a subalgebra of $A$, for all  $e \in \g_0$. Moreover, if $\g_0$ is finite, then each $A_e$ is unital with identity element $1_e$ and $1_A=\sum_{e \in \g_0}1_e$.

\subsection{Groupoids actions on algebras and Galois extension}
According to \cite{pd}, an {\it action of a groupoid $\g$ on an algebra $A$} is a pair $\beta=(E_g,\beta_{g})_{g\in \g}$,
where $E_g=E_{r(g)}$, $g\in \g$, is an ideal of $A$ and $\beta_{g}: E_{g^{-1}}\to E_g$ is an isomorphism of algebras, satisfying the following conditions:
\begin{itemize}
\item[(i)] $\beta_e = id_{E_e}$ is the identity map of $E_e$, for all $e\in \mathcal{G}_0$,
\item[(ii)] $\beta_{gh}(x)=\beta_g(\beta_h(x))$, for all $(g,h)\in \mathcal{G}_{d}\times \!_{r}\mathcal{G}$ and $x\in E_{h^{-1}}=E_{(gh)^{-1}}$.
\end{itemize}

In particular $\beta$ induces an action of the group $\g_e$ on the ideal $E_e$, for all $e \in \g_0$. 

\vu

The {\it{skew groupoid ring $A\ast_{\beta}\g$}} corresponding to $\beta$, introduced in \cite{pdd}, is the direct sum
\begin{align*}
A\ast_{\beta}\g=\oplus_{g\in \g}E_g\delta_g,
\end{align*}
where each symbol $\delta_g$ is a placeholder for the $g$-th component, with the usual addition and with the multiplication induced by the rule
\[
(x\delta_g)(y\delta_h)=
\left\{\begin{array}{rl}
x\beta_g(y)\delta_{gh}, & \mbox{ if } (g,h) \in \mathcal{G}_{d}\times \!_{r}\mathcal{G}\\
0, & \mbox{ otherwise,}
\end{array}
\right.
\]
for all $g,h \in \g$, $x \in E_g$ and $y \in E_h$. Clearly, this algebra is associative and $\g$-graded by construction.


\vu

Following \cite{pd} , the {\it{subalgebra of invariants}} of $A$ under $\beta$ is the set $$A^{\beta}=\{x \in A: \beta_g(x1_{g^{-1}})=x1_{g}, \forall g \in \g \}.$$ Moreover, $A$ is a {\it $\beta$-Galois extension} of $A^{\beta}$ if there exist elements $x_i,y_i \in A$, $1 \leq i \leq m$, such that $\sum_{1 \leq i \leq m} x_i \beta_g(y_i1_{g^{-1}}) = \delta_{e,g}1_e$, for all $e \in \g_0, g \in \g$. The set $\{x_i,y_i\}_{1 \leq i \leq m}$ is called a {\it Galois coordinate system} of $A$ over $A^{\beta}$.

\vu

\section{$\g$-sets and smash products}
\label{smash} According to \cite{PaqTam}, an \emph{action of a groupoid $\g$ on a set $X$}  is a pair $\alpha=( X_g, \alpha_g)_{g \in \mathcal{G}}$, where $X_g=X_{r(g)}$, $g \in \mathcal{G}$, is a subset of $X$ and $\alpha_g: X_{g^{-1}} \to X_g$ is a bijection satisfying the following conditions:
\begin{itemize}
\item[(i)] $\alpha_e= {id}_{X_e}$, for all $e \in \g_0$,
\item[(ii)] $\alpha_g(\alpha_h(x))=\alpha_{gh}(x)$, for all $(g,h) \in \mathcal{G}_{d}\times \!_{r}\mathcal{G}$ and $x \in X_{h^{-1}}=X_{(gh)^{-1}}$.
\end{itemize}
In this case, $X$ is said to be a {\it{$\g$-set}}. Observe that the action $\alpha$ induces an action of the group $\g_e$ on the set $X_e$, for all $e \in \g_0$. If $X=\coprod_{e \in \g_0} X_e$ (disjoint union), then $X$ is said to be a {\it split $\g$-set}. 

Let $X$ be a $\g$-set via $\alpha=(X_g, \alpha_g)_{g \in \g}$ and $Y \subseteq X$. The subset $Y$ is \emph{$\alpha$-invariant} if $\alpha_g(Y \cap X_{g^{-1}}) \subseteq Y \cap X_g$, for all $g \in \g$.

\vu

\begin{ex}\label{ex1}
Every groupoid $\g$ is a split $\g$-set via $\alpha=(R_{r(g)}, \alpha_g)_{g \in \g}$, where $\alpha_g(h)=gh$, for all $g \in \g$ and $h \in R_{d(g)}$.
\end{ex}

\begin{ex}\label{ex2}
Every groupoid $\g$ is a split $\g$-set via $\beta=(D_{r(g)}, \beta_g)_{g \in \g}$, where $\beta_g(h)=hg^{-1}$, for all $g \in \g$ and $h \in D_{d(g)}$.
\end{ex}

\begin{ex} \label{nosplit}
Let $\g$ be a groupoid, $X$ a $\g$-set via $\alpha=(X_g, \alpha_g)_{g \in \g}$ and $\beta=(E_g, \beta_g)_{g \in \g}$ an action of $\g$ on the algebra $A$. It is straightforward to verify that $Y=\{\tau \in Map(X,A): \tau(X_g) \subseteq E_g, \forall g \in \g \}$ is a $\g$-set via $\gamma=(Y_g,\gamma_g)_{g \in \g}$, where $Y_g=\{\tau \in Y: \tau(X_h)=0, \forall X_h \neq  X_g \}$ and 
\begin{align*}
\gamma_g(\tau)(x) &= \left\{\begin{array}{rl}
\beta_g \tau \alpha_{g^{-1}}(x),&\mbox{ if } x \in X_g\\
0,&\mbox{ otherwise.}
\end{array}
\right.
\end{align*}
However, $Y$ is not split, since $Y_e \cap Y_f=\{0\}$, for all $e, f \in \g_0$.
\end{ex}


Let $X$ be a split $\g$-set via $\alpha=( X_g, \alpha_g)_{g \in \mathcal{G}}$. Since $X$ is a split $\g$-set, given $x \in X$, there exists $e \in \g_0$ such that $x \in X_e$. The {\it $\g$-orbit of} $x$ is defined by $o(x)=\{\alpha_g(x): g\in D_e\}$. Moreover, the set $O^{\g}=\{ o(x) : x \in X \}$ is called \emph{$\g$-orbit of} $X$.
Now define on $X$  the following relation: 
\begin{align}\label{rel_orbit}
	x \approx_{\g} y \,\,\, \mbox{  if and only if  } \,\,\,  y \in o(x),
\end{align}which is clearly an equivalence relation.



Let $\g$ be a groupoid such that $\g_0$ is finite, $A$ be a $\g$-graded algebra and $X$ be a split $\g$-set via $\alpha=( X_g, \alpha_g)_{g \in \mathcal{G}}$. We define the {\it{smash product}}, denoted by $A \#_{\alpha}^{\g} X$, as the direct sum 
\begin{align}\label{definsmash}
A \#_{\alpha}^{\g} X = \oplus_{g \in \g} \oplus_{x \in X_{d(g)}} A_g \delta_x,
\end{align}
where the $\delta_g$' s are symbols, with the multiplication induced by
\[
(a_g \delta_x)(b_h \delta_y) = 
\left\{ \begin{array}{rl}
a_gb_h \delta_y, & \text{if } (g,h) \in \mathcal{G}_{d}\times \!_{r}\mathcal{G} \mbox{ and } \alpha_h(y)=x\\
0, & \text{otherwise,}
\end{array}
\right.
\]
for all $g$, $h \in \g$, $a_g \in A_g$, $b_h \in A_h$ and $x \in X_{d(g)}$, $y \in X_{d(h)}$. If $d(g)=r(h)$, then $a_g b_h \in A_{gh}$, $x \in X_{d(g)} =X_{r(h)}$ and $y \in  X_{d(h)}=X_{d(gh)}$. Thus this multiplication is well-defined. When $\g$ is, in particular, a group, this definition coincides with the one given in \cite{smash}.

\vu

It is straightforward to see that $A \#_{\alpha}^{\g} X$ is an associative algebra with unity $1_{A \#_{\alpha}^{\g} X}=\sum_{e \in \g_0}\sum_{x \in X_e}1_e\delta_{x}$. Moreover, the map $\eta: A \to A \#_{\alpha}^{\g} X$, given by 
\begin{align} \label{imersao}
\eta(a_g)=\sum_{e \in \g_0}\sum_{x \in X_e}a_g1_e\delta_x=\sum_{x \in X_{d(g)}}a_g\delta_x,
\end{align}
is a monomorphism of algebras. Notice that $\{ 1_e\delta_x : e \in \g_0, \mbox{ } x \in X_e \}$ is a set of orthogonal idempotents.

\vu
The map $\eta$ induces a structure of left $A$-module over $A \#_{\alpha}^{\g} X$ via
\begin{align*}
a_g \cdot (b_h\delta_y)=\eta(a_g)(b_h\delta_y)= \left\{
\begin{array}{rl}
a_gb_h \delta_y, & \text{if } d(g)=r(h),\\
0, & \text{otherwise,}
\end{array} \right.
\end{align*}
for all $a_g \in A$ and $b_h\delta_y \in A \#_{\alpha}^{\g} X$.
Analogously $A \#_{\alpha}^{\g} X$ is a right $A$-module via
\begin{align*}
(b_h\delta_y) \cdot a_g  =(b_h\delta_y)\eta(a_g)= \left\{
\begin{array}{rl}
b_h a_g \delta_{\alpha_{g^{-1}}(y)}, & \text{if } d(h)=r(g),\\
0, & \text{otherwise,}
\end{array} \right.
\end{align*}
for all $a_g \in A$ and $b_h\delta_y \in A \#_{\alpha}^{\g} X$.  Moreover, $A \#_{\alpha}^{\g} X$ is a $A$-bimodule.

In particular, we have the following equalities:
\begin{align}\label{eta}
\begin{aligned}
&a_g\delta_x=\eta(a_g)(1_{d(g)}\delta_x), \text{ for all } a_g \delta_x \in A \#_{\alpha}^{\g} X,\\
& a_g\delta_{\alpha_{g^{-1}}(x)}=(1_{r(g)}\delta_x)\eta(a_g), \text{ for all } a_g \in A, x \in X_g.
\end{aligned}
\end{align}

Let $X$ and $Z$ be $\g$-sets via the actions $\alpha = (X_g, \alpha_g)_{g \in \g}$ and $\beta=(Z_g, \beta_g)_{g \in \g}$, respectively. According to \cite{PaqTam}, the map $\varphi: X \to Z$ is a {\it morphism of $\g$-sets} if the following conditions hold, for all $g \in \g$:
\begin{itemize}
\item[(i)] $\varphi(X_g) \subseteq Z_g$,
\item[(ii)] $\varphi(\alpha_g(x)) = \beta_g(\varphi(x))$, for all $x \in X_{g^{-1}}$.
\end{itemize} Moreover, if $\varphi$ is a injective (resp. surjective) then $\varphi$ is a {\it monomorphism} (resp. {\it epimorphism}) of $\g$-sets. Finally, $\varphi$ is a {\it isomorphism} of $\g$-sets if $\varphi$  is bijective.

The next theorem shows that if $\varphi: X \to Z$ is a morphism of $\g$-sets, then it induces a morphism of algebras between the correspondents smash products. For this purpose, we introduce the notation $I_{g,z}=\{x\! \in \! X_{d(g)}: \varphi(x) \! =z\}$, for all $g \in \g$ and $z \in Z$.

\begin{teo} \label{morfismosmash}
Let $X$ and $Z$ be finite split $\g$-sets via $\alpha = (X_g, \alpha_g)_{g \in \g}$ and $\beta = (Y_k , \beta_k )_{k \in \K}$, respectively, and $A$ a $\g$-graded algebra. If $\varphi: X \to Z$ is a morphism of $\g$-sets, then $\varphi^*: A \#_{\beta}^{\g} Z \to A \#_{\alpha}^{\g} X$, given by
\[
\varphi^*(a_g\delta_z)=
\left\{ \begin{array}{rl}
\sum\limits_{x \in I_{g,z}}a_g\delta_x, & \mbox{ if }  I_{g,z} \neq \emptyset \\
0, & \mbox{ otherwise,}
\end{array}
\right.
\]
for all $a_g\delta_z \in A \#_{\beta}^{\g} Z$, is a morphism of algebras. Moreover, if $\varphi$ is injective (resp. surjective), then $\varphi^*$ is surjective (resp. injective).
\end{teo}

\begin{proof}
Take $a_g\delta_y, b_h\delta_z \in A \#^{\g}_{\beta} Z$ such that $d(g)=r(h)$. Suppose that $I_{g,y} \neq \emptyset$ and $I_{h,z} \neq \emptyset$. Then
\begin{align*}
\varphi^*(a_g\delta_y)\varphi^*(b_h\delta_z)=\sum\limits_{x \in I_{g,y}} \sum\limits_{w \in I_{h,z}} (a_g\delta_x)(b_h\delta_w).
\end{align*}
Notice that $\alpha_h(w) \in I_{g,y}$ if and only if $y=\beta_h(z)$. In fact, if $\alpha_h(w) \in I_{g,y}$, then $y=\varphi(\alpha_h(w))=\beta_h(\varphi(w))=\beta_h(z)$. Conversely, observe that $\alpha_h(w) \in X_{r(h)}=X_{d(g)}$. If $y=\beta_h(z)$, then $y=\beta_h(\varphi(w))=\varphi(\alpha_h((w))$, and consequently $\alpha_h(w) \in I_{g,y}$. Therefore
\begin{align*}
\varphi^*(a_g\delta_y)\varphi^*(b_h\delta_z)&=
\left\{ \begin{array}{rl}
\sum\limits_{w \in I_{h,z}}(a_g\delta_{\alpha_h(w)})(b_h\delta_w), & \mbox{ if } \alpha_h(w) \in I_{g,y} \\
0, & \mbox{ otherwise}
\end{array} \right.\\
&=
\left\{ \begin{array}{rl}
\sum\limits_{w \in I_{gh,z}}a_gb_h\delta_w, & \mbox{ if } \beta_h(z)=y \\
0, & \mbox{ otherwise}
\end{array} \right.
 \\
&= \varphi^*((a_g\delta_y)(b_h\delta_z)).
\end{align*}
Clearly, if $I_{h,z} = \emptyset$ or $d(g) \neq r(h)$, then $\varphi^*(a_g\delta_y)\varphi^*(b_h\delta_z)=0=\varphi^*((a_g\delta_y)(b_h\delta_z))$. Finally, if $I_{g,y}= \emptyset$ and $I_{h,z} \neq \emptyset$, then $\beta_h(z) \neq y$. In fact, if $\beta_h(z)=y$, then $y=\beta_h(\varphi(w'))=\varphi(\alpha_h(w'))$, which is a contradiction.
\end{proof}

\section{Duality Theorem} 

In this section we introduce the notion of $(\g_{\alpha},\K_{\beta})$-set,  which allows us to define the skew groupoid ring $(A\#_{\alpha}^{\g} X)*_{\gamma} \K$, which is one of the main objects of study of this paper. Throughout this section, $A$ denotes a $\g$-graded algebra.

\vu

\begin{defin}\label{def37}
Let $\g$ and $\K$ be groupoids. The set $X$ is a {\it{$(\g_{\alpha},\K_{\beta})$-set}} if $X$ is a $\g$-set via $\alpha = (X_g, \alpha_g)_{g \in \g}$, $X$ is a $\K$-set via $\beta = (Y_k , \beta_k )_{k \in \K}$ and the following conditions are  satisfied,  for all $g \in \g$ and $k \in \K$:
\begin{itemize}
\item[(i)] $\alpha_g(X_{g^{-1}} \cap Y_k) \subseteq X_g \cap Y_k$  and $\beta_k(X_g \cap Y_{k^{-1}}) \subseteq X_g \cap Y_k$,
\item[(ii)] $\alpha_g(\beta_k (x))= \beta_k(\alpha_g(x))$, for  all $x \in X_{g^{-1}} \cap Y_{k^{-1}}$.
\end{itemize}
Moreover, if $X$ is a split $\g$-set and a split $\K$-set, then $X$ is said to be a {\it{split $(\g_{\alpha},\K_{\beta})$-set}}.
\end{defin}

Observe that item (i)  means that $Y_k$ is $\alpha$-invariant, for all $k \in \K$, and $X_g$ is $\beta$-invariant, for all $g \in \g$.

\begin{ex}
Every groupoid $\g$ is a split $(\g_{\alpha},\g_{\beta})$-set via the actions $\alpha$ and $\beta$ given in Examples \ref{ex1} and \ref{ex2}, respectively.
\end{ex}

\begin{lema} \label{xkgconj}
Let $X$ be a $(\g_{\alpha}, \K_{\beta})$-set via $\alpha = (X_g, \alpha_g)_{g \in \g}$ and $\beta = (Y_k , \beta_k )_{k \in \K}$. Then, $Y_k$ is a $\g$-set and $\beta_k: Y_{k^{-1}} \to Y_k$ is a isomorphism of $\g$-sets, for all $k \in \K$. Furthermore, if $X$ is a split $\g_{\alpha}$-set, then $Y_k$ is a split $\g$-set. 
\end{lema}

\begin{proof} Fix $k \in \K$. For each $g \in \g$, $\alpha_g(Y_k \cap X_{g^{-1}}) \subseteq Y_k \cap X_g$. So, taking $Z^{k}_g= Y_k \cap X_g$ and $\theta^k_g=\alpha_g|_{Z^{k}_{g^{-1}}}$ is easy to see that $\theta^k=(Z^k_g,\theta^k_g)_{g \in \g}$ determines an action of $\g$ on $Y_k$. Now, observe that $\beta_k$ is a morphism of $\g$-sets. Indeed, $\beta_{k}(Z^{k^{-1}}_g)=\beta_{k}(Y_{k^{-1}} \cap X_g) \subseteq Y_k \cap X_g=Z^k_g$, for all $g \in \g$. Moreover, $\beta_{k}(\theta^{k^{-1}}_g(x))=\beta_{k}(\alpha_g(x))=\alpha_g(\beta_{k}(x))=\theta^{k}_g(\beta_{k}(x))$, for all $g \in \g$ and $x \in Z^{k^{-1}}_{g^{-1}}$.  
\end{proof}

In the following proposition we present an action $\gamma$ of the groupoid $\K$ on the algebra $A \#_{\alpha}^{\g} X$. Thus we can consider the skew groupoid ring $(A \#_{\alpha}^{\g} X)*_{\gamma} \K $. 

\vu

\begin{prop} \label{skew}
Let $\g$ and $\K$ be groupoids such that $\g_0$ is finite, $X$ a finite split $(\g_{\alpha}, \K_{\beta})$-set via $\alpha = (X_g, \alpha_g)_{g \in \g}$ and $\beta = (Y_k , \beta_k )_{k \in \K}$ and $A$ a $\g$-graded algebra. Then, there exists an action $\gamma=(E_k,\gamma_k)_{k \in \K}$ of $\K$ on $A \#_{\alpha}^{\g} X$ such that $E_k$ is unital, for all $k \in \K$, and $A \#_{\alpha}^{\g} X=\oplus_{p \in \K_0}E_p$.
\end{prop}

\begin{proof} 
Since $X$ is a finite split $\g$-set, then by Lemma \ref{xkgconj} $Y_k$ is a finite split $\g$-set, for all $k \in \K$. Thus
\begin{align*}
A \#_{\theta^{r(k)}}^{\g} Y_{r(k)}=A \#_{\theta^k}^{\g} Y_k=\oplus_{g \in \g} \oplus_{x \in Z_{d(g)}^k}A_g\delta_x= \oplus_{g \in \g} \oplus_{x \in Y_k \cap X_{d(g)}}A_g\delta_x,
\end{align*}
is an ideal of $A \#_{\alpha}^{\g} X$ with unity given by ${\bf 1}_k=\sum_{e \in \g_0}\sum_{x \in Y_k \cap X_e}1_e\delta_x$. Indeed, let $a_g \delta_x \in A \#_{\alpha}^{\g} X$, $x \in X_{d(g)}$, and $b_h \delta_y \in A \#_{\theta^k}^{\g} Y_k$, $y \in Y_k \cap X_{d(h)}$. If $d(h)=r(g)$ and $\alpha_g(x)=y$, then $(b_h\delta_y)(a_g\delta_x)=b_ha_g\delta_x$.
Hence,  $x=\alpha_{g^{-1}}(y) \in \alpha_{g^{-1}}(Y_k \cap X_{r(g)}) \subseteq Y_k \cap X_{d(g)}= Y_k \cap X_{d(hg)}$. If $d(g) \neq r(h)$ or $\alpha_g(x) \neq y$, then $(b_h\delta_y)(a_g\delta_x)=0$. So, $A \#_{\theta^k}^{\g} Y_k$ is a right ideal of $A \#_{\alpha}^{\g} X$. Similarly, $A \#_{\theta^k}^{\g} Y_k$ is a left ideal of $A \#_{\alpha}^{\g} X$.

\vu

By Lemma \ref{xkgconj}, $\beta_{k^{-1}}: Y_k \to Y_{k^{-1}}$ is an isomorphism of $\g$-sets, for all $k \in \K$. By Theorem \ref{morfismosmash}, $\beta^*_{k^{-1}}: A \#_{\theta^{k^{-1}}}^{\g}Y_{k^{-1}} \to A \#_{\theta^k}^{\g} Y_k$ is an isomorphism of algebras. Taking $E_k=A \#_{\theta^k}^{\g} Y_k$ and $\gamma_k=\beta^*_{k^{-1}}$, for all $k \in \K$, is easy to see that $\gamma=(E_k,\gamma_k)_{k \in \K}$ is an action of $\K$ on $A \#_{\alpha}^{\g} X$ such that $A \#_{\alpha}^{\g} X = \oplus_{p \in \K_0} E_p$.
\end{proof}

 


Now, we present a theorem of duality in the sense of \cite{a1}. Precisely, we show that there exists an isomorphism of algebras from the skew groupoid ring $(A\#_{\alpha}^{\g} X)*_{\gamma} \K$ to a particular algebra of endomorphisms. The following results describe this algebra.

\begin{lema} \label{orbita}
If $X$ is a $(\g_{\alpha},\K_{\beta})$-set via $\alpha = (X_g, \alpha_g)_{g \in \g}$ and $\beta = (Y_k , \beta_k )_{k \in \K}$, such that $X$ is a split $\K$-set, then $O^{\K}$ is a $\g$-set and $\varphi: X \to O^{\K}$, given by $\varphi(x)=o(x)$, for all $x \in X$, is an epimorphism of $\g$-sets.
\end{lema}

\begin{proof} 
Let $x \in Y_p$, for some $p \in \K_0$. Recall that $o(x)=\{\beta_k(x): k \in D_p\}$. 
For each $g \in \g$, define $O_g^{\K} = \{o(x): x \in X_g \}$ and $\lambda_{g}: {O_{g^{-1}}^{\K}} \to O_g^{\K}$, such that $\lambda_g(o(x))=o(\alpha_g(x))$, for all $o(x) \in {O_{g^{-1}}^{\K}}$. Let $x,y\in X_{g^{-1}}$, then there exist unique $p,q \in \K_0$ such that $x \in Y_{p}$ and $y \in Y_{q}$. Suppose that $o(x)=o(y)$ in ${O_{g^{-1}}^{\K}}$, then $x=\beta_k(y)$, $k \in D_q$. Since $y \in X_{g^{-1}} \cap Y_{q}=X_{g^{-1}} \cap Y_{d(k)}=X_{g^{-1}} \cap Y_{k^{-1}}$, then $\alpha_g(x)=\alpha_g(\beta_k(y))=\beta_k(\alpha_g(y)) \in o(\alpha_g(y))$. Therefore $\lambda_g$ is well defined, for all $g \in \g$. Clearly, $\lambda = ({O_g^{\K}}, \lambda_g )_{g \in \g}$ is an action of $\g$ on $O^{\K}$ and $\varphi$ is an epimorphism of $\g$-sets.
\end{proof}

\vu

By Lemma \ref{orbita} and Theorem \ref{morfismosmash}, the map $\varphi^*: A \#_{\lambda}^{\g} O^{\K} \to A \#_{\alpha}^{\g} X$, given by $\varphi^*(a_g\delta_{o(x)})=\sum_{y \in I_{g,o(x)} }a_g\delta_y$, for all $a_g\delta_{o(x)} \in A \#_{\lambda}^{\g} O^{\K}$, is a monomorphism of algebras. In the following proposition, the image of $\varphi^*$ is described from  the action $\gamma$ of $\K$ on $A \#_{\alpha}^{\g} X$ given in Proposition \ref{skew}.

\vu

\begin{prop} \label{partefixa}
If $X$ is a $(\g_{\alpha},\K_{\beta})$-set via $\alpha = (X_g, \alpha_g)_{g \in \g}$ and $\beta = (Y_k , \beta_k )_{k \in \K}$, such that $X$ is a split $\K$-set, and $A$ is a $\g$-graded algebra, then $$\varphi^*(A \#_{\lambda}^{\g} O^{\K})=(A \#_{\alpha}^{\g} X)^{\gamma}.$$
\end{prop}

\begin{proof}
Let $a_g\delta_{o(x)} \in A \#_{\lambda}^{\g} O^{\K}$. For all $k \in \K$
\begin{align*}
\gamma_k(\varphi^*(a_g\delta_{o(x)}){\bf 1}_{k^{-1}}) &=\gamma_k \bigg(\bigg(\sum\limits_{y \in I_{g,o(x)}} a_g\delta_y \bigg) \bigg(\sum\limits_{e \in \g_0}\sum\limits_{z \in Y_{k^{-1}} \cap X_e}1_e\delta_{z} \bigg)\bigg) \\
&=\gamma_k \bigg(\sum\limits_{y \in I_{g,o(x)}}\sum\limits_{z \in  Y_{k^{-1}} \cap X_{d(g)}} (a_g\delta_y)(1_{d(g)}\delta_{z})\bigg)\\
&=\gamma_k \bigg(\sum\limits_{y \in I_{g,o(x)} \cap Y_{k^{-1}} }a_g\delta_y \bigg) \\
&=\sum\limits_{y \in I_{g,o(x)} \cap Y_{k^{-1}} }a_g\delta_{\beta_k(y)}.
\end{align*}

On the other hand,
\begin{align*}
\varphi^*(a_g\delta_{o(x)}){\bf 1}_{k} &= \bigg( \sum\limits_{y \in I_{g,o(x)}} a_g\delta_y \bigg)\bigg(\sum\limits_{e \in \g_0}\sum\limits_{z \in Y_k \cap X_e}1_e\delta_{z} \bigg) = \sum\limits_{z \in I_{g,o(x)} \cap Y_k } a_g\delta_z.
\end{align*}
Since $\beta_k: Y_{k^{-1}} \to Y_k$ is a bijection and $z \in Y_k \cap X_{d(g)}$, then there exists unique $y \in Y_{k^{-1}}$ such that $\beta_k(y)=z$, so $o(z)=o(y)$. Moreover, as $X$ is a $(\g_{\alpha},\K_{\beta})$-set, then $y=\beta_{k^{-1}}(z) \in \beta_{k^{-1}}(Y_k \cap X_{d(g)}) \subseteq Y_{k^{-1}} \cap X_{d(g)}$. Hence $\varphi^*(a_g\delta_{o(x)}){\bf 1}_{k}=\gamma_k(\varphi^*(a_g\delta_{o(x)}){\bf 1}_{k^{-1}})$, for all $k \in \K$. Therefore $\varphi^*(A \#_{\lambda}^{\g} O^{\K}) \subseteq (A\#_{\alpha}^{\g} X)^{\gamma}$.

\vu

Conversely let $w=\sum_{g \in \g} \sum_{x \in X_{d(g)}}a^x_g\delta_x \in (A \#_{\alpha}^{\g} X)^{\gamma}$, where $a_g^x \in A_g$, for all $g \in \g$ and $x \in X_{d(g)}$. Since $\gamma_k(w{\bf 1}_{k^{-1}})=w{\bf 1}_k$, for all $k \in \K$, then
\begin{align*}
\sum\limits_{g \in \g} \sum\limits_{y \in Y_{k^{-1}} \cap X_{d(g)}}a^y_g\delta_{\beta_k(y)}=\sum\limits_{g \in \g} \sum\limits_{z \in Y_k \cap X_{d(g)}}a^z_g\delta_z.
\end{align*}
If $z=\beta_k(y)$, then $a^y_g=a^z_g$, for all $g \in \g$.  For each $g \in \g$, let $T_{d(g)}=T \cap X_{d(g)}$, where $T$ is a transversal of the equivalence relation $\approx_{\K}$ defined in \eqref{rel_orbit}. Let $x \in I_{g,o(y)}$, where $y \in T_{d(g)}$. Then $x=\beta_t(y)$, for some $t \in \K$ such that $y \in Y_{t^{-1}}$ and $a^{y}_g=a^x_g=a^{\beta_t(y)}_g$. Take $\bar{w}=\sum_{g \in \g}\sum_{y \in T_{d(g)}}a^{y}_g\delta_{o(y)}  \in A \#_{\lambda}^{\g} O^{\K}$. 
Note that for each $y \in T_{d(g)}$, we have that $\varphi^*(a_g^{y}\delta_{o(y)})=\sum_{x \in I_{g,o(y)}}a_g^{y}\delta_x=\sum_{x \in I_{g,o(y)}}a_g^x\delta_x$. Whence
\begin{align*}
\varphi^*(\bar{w})= \sum\limits_{g \in \g}\sum\limits_{y \in T_{d(g)}} \varphi^*(a^{y}_g\delta_{o(y)})=\sum\limits_{g \in \g}\sum\limits_{y \in T_{d(g)}} \sum\limits_{x \in I_{g,o(y)}}a_g^x\delta_x = \sum\limits_{g \in \g} \sum\limits_{x \in X_{d(g)}}a^x_g\delta_x = w,
\end{align*}
since $X_{d(g)}=\coprod_{y \in T_{d(g)}}I_{g,o(y)}$. Therefore $(A \#_{\alpha}^{\g} X)^{\gamma}=\varphi^*(A \#_{\lambda}^{\g} O^{\K})$.
\end{proof}

Consequently, $A \#_{\lambda}^{\g} O^{\K} \simeq \varphi^*(A \#_{\lambda}^{\g} O^{\K}) =(A \#_{\alpha}^{\g} X)^{\gamma}$.

\begin{defin}
Let $\mathcal{L}$ be a groupoid and $X$ a $\mathcal{L}$-set via $\vartheta = (X_l,\vartheta_l)_{l \in \mathcal{L}}$. The action $\vartheta$ is fully faithful if $\vartheta_l(x)=x$, for some $x \in X_l \cap X_{l^{-1}}$, implies that $l \in \mathcal{L}_0$.
\end{defin}

Let $\vartheta$ be an action fully faithful of  $\mathcal{L}$ on $X$. If $\mathcal{L}_0$ and $X$ are finite, then it is easy to verify that $\mathcal{L}$ is finite. 

\begin{teo} \label{dualidade}
Let $\g$ and $\K$ be groupoids such that $\g_0$ and $\K_0$ are finite, $X$ a finite split $(\g_{\alpha}, \K_{\beta})$-set and $A$ a $\g$-graded algebra. If the action of $\K$ on $X$ is fully faithful, then 
\begin{align*}
(A \#_{\alpha}^{\g} X)*_{\gamma} \K \simeq End({A \#_{\alpha}^{\g} X}_{A \#_{\lambda}^{\g} O^{\K}}).
\end{align*}
\end{teo}
\begin{proof}
We show that the extension $(A \#_{\alpha}^{\g}X)^{\gamma} \subseteq A \#_{\alpha}^{\g} X$ is Galois. Define $u_e=\sum_{x \in X_e}1_e\delta_x$, for all $e \in \g_0$. We proof that $\{u_e,u_e\}_{e \in \g_0}$ is a Galois coordinate system of $A \#_{\alpha}^{\g} X$ over $(A \#_{\alpha}^{\g} X)^{\gamma}$.
Consider $k \in \K \setminus \K_0$, then
\begin{align*}
\sum\limits_{e \in \g_0}\sum\limits_{x \in X_e}(1_e\delta_x)\gamma_k((1_e\delta_x){\bf 1}_{k^{-1}})&=\sum\limits_{e \in \g_0}\sum\limits_{x \in X_e}(1_e\delta_x)\gamma_k[(1_e\delta_x)(\sum\limits_{{f \in \g_0}\atop{y \in X_f \cap Y_{k^{-1}}}}1_f\delta_{y} )]&\\
&= \sum\limits_{e \in \g_0}\sum\limits_{x \in X_e \cap Y_{k^{-1}}}(1_e\delta_x)\gamma_k(1_e\delta_x) \\
&=\sum\limits_{e \in \g_0}\sum\limits_{x \in X_e \cap Y_{k^{-1}}}(1_e\delta_x)(1_e\delta_{\beta_k(x)})=0,
\end{align*}

\noindent once if there exists $x \in X_e \cap Y_{k^{-1}}$ such that $x=\beta_k(x)$, then $k \in \K_0$, since the action is fully faithful.

\vu

On the other hand, if $p \in \K_0$, by a similar calculations
\begin{align*}
\sum\limits_{e \in \g_0}\sum\limits_{x \in X_e}(1_e\delta_x)\gamma_p((1_e\delta_x){\bf 1}_p)={\bf 1}_p.
\end{align*}

Hence the extension $(A \#_{\alpha}^{\g} X)^{\gamma} \subseteq A \#_{\alpha}^{\g} X$ is Galois. Therefore, by \cite[Theorem 5.3]{pd} and Proposition \ref{partefixa}, it follows that
\begin{align*}
(A \#_{\alpha}^{\g} X)*_{\gamma}\K \simeq End({A \#_{\alpha}^{\g} X}_{(A \#_{\alpha}^{\g} X)^{\gamma}}) \simeq End({A \#_{\alpha}^{\g} X}_{A \#_{\lambda}^{\g} O^{\K}}).
\end{align*}
\end{proof}

In the following, we present some consequences of this result.

%


\begin{cor}
Let $\g$ be a finite groupoid and $\mathcal{H}$ a wide subgroupoid of $\g$. If $A$ is a $\g$-graded algebra then 
\begin{align*}
(A \#_{\alpha}^{\g} \g)*_{\gamma} \mathcal{H} \simeq End(A \#_{\alpha}^{\g} \g_{A \#_{\lambda}^{\g} \g / \mathcal{H}}).
\end{align*}
\end{cor}

\begin{proof}
By Example \ref{ex1}, $\g$ is a split $\g$-set via $\alpha=(X_g,\alpha_g)_{g \in \g}$ where $X_g=R_{r(g)}$ and $\alpha_g: X_{g^{-1}} \to X_g$ is given by $\alpha_g(h)=gh$, $g \in \g$ and $h \in X_{g^{-1}}$.
On the other hand, $\g$ is a $\mathcal{H}$-set via $\beta=(Y_h,\beta_h)_{h \in \mathcal{H}}$ where $Y_h=D_{r(h)}$ and $\beta_h: Y_{h^{-1}} \to Y_h$ is given by $\beta_h(l)=lh^{-1}$, for all $h \in \mathcal{H}$ and $l \in Y_{h^{-1}}$, and this action is fully faithful. Since  $\mathcal{H}$ is a wide subgroupoid of $\g$, then $\g$ is a split $\mathcal{H}$-set. Is easy to see that $\g$ is a split $(\g_{\alpha}, \mathcal{H}_{\beta})$-set and $O^{\mathcal{H}}=\g / \mathcal{H}$. So the claim follows.
\end{proof}

Consider $X$ a set and denote by $I(X)$ the set of partial bijection between subsets of $X$. Is well-know that $I(X)$ is a groupoid with restricted product. For more details one can see \cite{defgrupoide}. Let $X$ be a $\g$-set via  $\alpha=(X_g,\alpha_g)_{g \in \g}$, denote by $I_{\g}(X)$ the groupoid which the elements are partial bijections $\rho \in I(X)$ such that $dom(\rho),im(\rho)$ are $\alpha$-invariants and  $\rho$ is a morphism of $\g$-set.


\begin{defin}
Let $X$ be a split $\g$-set via  $\alpha=(X_g,\alpha_g)_{g \in \g}$. The action $\alpha$ is transitive if for any $x,y \in X$ there exists $g \in \g$ such that $\alpha_g(x)=y$. 	
\end{defin}

\begin{cor}
Let $\g$ be a groupoid such that $\g_0$ is finite, $X$ a finite split $\g$-set via  $\alpha=(X_g,\alpha_g)_{g \in \g}$ such that  $\alpha$ is transitive and $A$ a $\g$-graded algebra. Then
\begin{align*}
	(A \#_{\alpha}^{\g} X) *_{\gamma} I_{\g}(X) \simeq End(A \#_{\alpha}^{\g} X_{A \#_{\lambda}^{\g} O}),
\end{align*}
where $O$ is the set of the orbits of $X$ by the action of $I_{\g}(X)$.
\end{cor}
 \begin{proof}
 Observe that $X$ is a split $(\g_{\alpha}, I_{\g}(X)_{\beta})$-set where \begin{align*}\beta=( Y_{\rho}, \beta_{\rho})_{ \rho \in I_{\g}(X)},
 \end{align*} with $Y_{\rho}=im(\rho)$ and $\beta_{\rho}=\rho$, for all $\rho \in I_{\g}(X)$. It remains to show that the action of $I_{\g}(X)$ on $X$ is fully faithful. Indeed, let $x \in Y_{\rho^{-1}} \cap Y_{\rho}$ such that $\rho(x)=x$, for some $\rho \in I_{\g}(X)$. Since the action of $\g$ on $X$ is transitive, for all $y \in dom(\rho)$, there exists $g \in \g$ such that $y=\alpha_g(x)=\alpha_g(\rho(x))=\rho(\alpha_g(x))=\rho(y)$. Hence $\rho \in I_{\g}(X)_0$.
 \end{proof}

\section{Morita Theory}

For the results of this section we consider $\g$ a groupoid such that $\g_0$ is finite, $A$ a $\g$-graded algebra and $X$ a finite split $\g$-set via  $\alpha=(X_g,\alpha_g)_{g \in \g}$. A left $A$-module $M$ is said a {\it{graded $A$-module of type $X$}}, or simply {\it{$X$-graded}}, if there exists a family $\{M_x\}_{x \in X}$ of vector subspaces of $M$ such that $M=\oplus_{x \in X}M_x$, satisfying
\begin{align}\label{xgraded}
A_gM_x\left\{
\begin{array}{rl}
\subseteq M_{\alpha_g(x)},& \mbox{if } x \in X_{d(g)}\\
=0,& \mbox{otherwise},
\end{array}
\right.
\end{align}
for all $g \in \g$ and $x \in X$. Similarly, we can define graded right $A$-module of type $X$.
Denote by $(\g,X,A)_l$ (resp. $(\g,X,A)_r$) the category of the left (resp. right) $A$-modules $X$-graded.  A morphism in $(\g,X,A)_l$ (resp. $(\g,X,A)_r$) is a morphism of left (resp. right) $A$-modules $\phi: M \to N$  such that $\phi(M_x) \subseteq N_x$, for all $x \in X$. For simplicity we denote the category of the left modules by  $(\g,X,A)$. 


\begin{teo} \label{isocategorias}
The categories $(\g,X,A)$ and $_{A \#_{\alpha}^{\g} X}\mathcal{M}$ are isomorphic.
\end{teo}

\begin{proof}
We define functors $F: \, _{A \#_{\alpha}^{\g} X}\mathcal{M} \to (\g,X,A)$ and $G: (\g,X,A) \to \, _{A \#_{\alpha}^{\g} X}\mathcal{M}$ by
\begin{align*}
F(V)=V, \,\,\,\, F(\psi)=\psi, \,\,\,\, G(M)=M, \,\,\,\, \mbox{and} \,\,\,\, G(\phi)=\phi,
\end{align*}
for all objects $V \in \, _{A \#_{\alpha}^{\g} X}\mathcal{M}$, $M \in (\g,X,A)$, and for all morphisms $\psi \in \, _{A \#_{\alpha}^{\g} X}\mathcal{M}$, $\phi \in (\g,X,A)$. If $V$ is a left $A \#_{\alpha}^{\g} X$-module, then $V$ is a left $A$-module via $ a \vartriangleright v=\eta(a)v$, for all  $a \in A$ and $v \in V$, where $\eta$ is given in \eqref{imersao}. Moreover, $V$ is $X$-graded. In fact, for all $x \in X$, we have that  $x \in X_e$, for some $e \in \g_0$. Define $V_x=(1_e\delta_x)V$. Then, for all $v \in V$, it follows that
\begin{align*}
v=1_{A \#_{\alpha}^{\g} X}\,v= \bigg(\sum\limits_{e \in \g_0} \sum\limits_{x \in X_e}1_e \delta_x \bigg)v= \sum\limits_{e \in \g_0} \sum\limits_{x \in X_e}(1_e \delta_x)v \in \sum\limits_{x \in X}V_x.
\end{align*}
Since $C = \{ 1_e\delta_x : e \in \g_0, \mbox{ } x \in X_e \}$ is the set of orthogonal idempotents, we have that $V=\oplus_{x \in X}V_x$. Now we verify the condition \eqref{xgraded}. Let  $a_g \in A_g$, $g \in \g$ and $v \in V_x$, $x \in X$. Suppose that $x \in X_{d(g)}$. Hence $v=(1_{d(g)}\delta_x)v'$, for some $v' \in V$. Then
\begin{align*}
a_g \vartriangleright v&= \eta(a_g)(1_{d(g)}\delta_x)v'\overset{\eqref{eta}}{=} (a_g\delta_x)v' \\
&=(1_{r(g)}\delta_{\alpha_g(x)})\big((a_g\delta_x))v'\big) \in V_{\alpha_g(x)}.
\end{align*}
For the other hand, if $x \in X_e$, $e \neq d(g)$, then $a_g  \vartriangleright v=0$. 

Now, let $\psi: V \rightarrow W$ be a left $A \#_{\alpha}^{\g} X$-modules morphism. If $ (1_e\delta_x)v \in V_x$, then $$\psi((1_e\delta_x)v) = (1_e\delta_x)\psi(v) \in (1_e\delta_x)W = W_x.$$


Reciprocally, if $M$ is $X$-graded, then $M$ is a left $A \#_{\alpha}^{\g} X$-module via $(a_g\delta_x) \cdot m=a_gm_x$, for all $a_g\delta_x \in A \#_{\alpha}^{\g} X$ and $m \in M$. Clearly $1_{A \#_{\alpha}^{\g} X} \cdot m=m$, for all $m \in M$. Let $a_g\delta_x$, $b_h\delta_y \in A \#_{\alpha}^{\g} X$ and $m \in M$. Hence
\begin{align*}
\big((a_g\delta_x)(b_h\delta_y) \big)\cdot m & =\left\{
\begin{array}{rl}
(a_gb_h\delta_y) \cdot m,& \mbox{if } d(g)=r(h) \mbox{ and } \alpha_h(y)=x\\
0, & \mbox{otherwise}
\end{array}
\right. \\
&=\left\{
\begin{array}{rl}
a_gb_hm_y, & \mbox{if } d(g)=r(h) \mbox{ and } \alpha_h(y)=x\\
0, & \mbox{otherwise}
\end{array}
\right. \\
&=\left\{
\begin{array}{rl}
a_g(b_hm_y)_x,& \mbox{if } d(g)=r(h) \mbox{ and } \alpha_h(y)=x\\
0, & \mbox{otherwise}
\end{array}
\right. \\
&=(a_g\delta_x) \cdot (b_hm_y)\\&= (a_g\delta_x) \cdot  \big( (b_h\delta_y) \cdot m \big).
\end{align*}

Let $\phi: M \rightarrow N$ be a morphism in $(\g,X,A)$, $a_g\delta_x \in A \#_{\alpha}^{\g} X$ and $m \in M$ such that $\phi(m) = n$. Then,
\begin{align*}
\phi((a_g\delta_x) \cdot m)  = \phi(a_gm_x) = a_g\phi(m_x) = a_gn_x = (a_g\delta_x) \cdot n = (a_g\delta_x) \cdot \phi(m).
\end{align*} 
Finally, consider $V \in \, _{A \#_{\alpha}^{\g} X}\mathcal{M}$ and $v \in G(F(V))$. Hence, for all $a_g\delta_x \in A \#_{\alpha}^{\g} X$, follows that
\begin{align*}
(a_g\delta_x)\cdot v=\eta(a_g)v_x=\eta(a_g)(1_{d(g)}\delta_x)v\overset{\eqref{eta}}{=}(a_g\delta_x)v,
\end{align*}
where the last action is the usual action in $_{A \#_{\alpha}^{\g} X}\mathcal{M}$. Now let $M \in (\g,X,A)$, $a_g \in A_g$, for $g \in \g$, and $m_x \in (F(G(M)))_x$, for $x \in X$. Hence
\begin{align*}
a_g  \vartriangleright m_x &= \eta(a_g) \cdot m_x = \sum\limits_{y \in X_{d(g)}}a_g(m_x)_y \\
&=\left\{
\begin{array}{rl}
a_gm_x,& \mbox{if } x \in X_{d(g)}\\
0, & \mbox{ otherwise}
\end{array}
\right. \\
&=a_gm_x.
\end{align*}

\noindent It remains to show the compatibility between the graduations. Let $m \in (F(G(M)))_x$ such that $x \in X_e$, $e \in \g_0$. Then, $
m=(1_e\delta_x)m=1_em_x=m_x $
the usual graduation of $M \in (\g,X,A)$.
\end{proof}

\begin{obs} \label{isocategoriasdireita}
 The Theorem \ref{isocategorias} is also true for right modules, i.e., the categories $(\g,X,A)_r$ and  $\mathcal{M}_{A \#_{\alpha}^{\g} X}$ are isomorphic.
\end{obs}

Recall that $X$ is a split $\g$-set via $\alpha=(X_g,\alpha_g)_{g \in \g}$. Hence, for each $x \in X$, there exists a unique $e \in \g_0$ such that $x \in X_e$. The set $\g_x=\{g \in \g_e: \alpha_g(x)=x\}$ is said {\it{$x$-stabilizer in $\g$}}. Moreover, $A^{\g_x}=\oplus_{g \in \g_x}A_g$, for all $x \in X$. In the following, we show that the categories isomorphism given in the Theorem \ref{isocategorias} allows establish a condition for $A^{\g_x}$ and $A \#_{\alpha}^{\g} X$ be Morita equivalent.

Let $x,y \in X$ such that $x\in X_e$ and $y \in X_f$. Consider the set
\begin{align*}
V_{x,y}=\bigoplus\limits_{{h \in D_e \cap R_f} \atop {\alpha_h(x)=y}}A_h.
\end{align*}
Clearly, $V_{x,x}=A^{\g_x}$. If $V_{x,y} \neq \emptyset$ and $V_{y,x} \neq \emptyset$, then $V_{x,y}$ is a right $A^{\g_x}$-module and $V_{y,x}$ is a left $A^{\g_x}$-module via multiplication. Therefore, $\oplus_{y \in X}V_{x,y}$ is a right $A^{\g_x}$-module 
and $\oplus_{y \in X}V_{y,x}$ is a left $A^{\g_x}$-module. We set:
\begin{align*}
^xV:=\oplus_{y \in X}V_{x,y} \quad \text{and} \quad V^x:=\oplus_{y \in X}V_{y,x}.
\end{align*}

Notice that, for all $x \in X$, it is straightforward to see that $V^x$ is a left $A$-module by multiplication and $^xV$ is a right $A$-module by multiplication. Hence, by Theorem \ref{isocategorias} and Remark \ref{isocategoriasdireita}, it follows that $^xV$ is a left $A \#_{\alpha}^{\g} X$-module and $V^x$ is a right $A \#_{\alpha}^{\g} X$-module, via $(a_g\delta_y)\cdot w=a_gw_y$ and $v \cdot (a_g\delta_y)=(va_g)_y$, respectively, for all $a_g\delta_y \in A \#_{\alpha}^{\g} X$, $v \in$ $^xV$ and $w \in V^x$. Then we can conclude that $^xV$ is a $(A \#_{\alpha}^{\g} X, A^{\g_x})$-bimodule and $V^x$ is a $(A^{\g_x}, A \#_{\alpha}^{\g} X)$-bimodule. Now, define
\begin{align*}
& (\,,): V^x \otimes_{A \#_{\alpha}^{\g} X} {^x}V \to A^{\g_x}, \mbox{ by } (v,w)=\sum_{g \in \g_x}(vw)_g,\\
& [\,,]: {^xV} \otimes_{A^{\g_x}} V^x \to A \#_{\alpha}^{\g} X, \mbox{ by } [w,v]=\sum_{z \in X}(wv_z)\delta_z,
\end{align*}
for all $v=\sum_{z \in X}v_z \in V^x$, where $v_z= \sum\limits_{{h \in D_f \cap R_e} \atop {\alpha_h(z)=x}}a_h$, and $w \in  {^xV}$.

\begin{teo} \label{morita} Keeping the notation above,
\begin{itemize}
\item[(i)] $[A \#_{\alpha}^{\g} X, V^x, {^xV}, A^{\g_x}, [\, ,], (\, ,)]$ is a Morita context. 
\item[(ii)] If the Morita context is strict and $X_e \neq \emptyset$, for all $e \in \g_0$, then
$$\sum\limits_{{g \in \g}\atop{\alpha_g(y)=x}}A_{g^{-1}}A_g=\sum_{g \in \g}A_{d(g)}, \mbox { for all } y \in X.$$
\item[(iii)] If $\sum\limits_{{g \in \g}\atop{\alpha_g(y)=x}}A_{g^{-1}}A_g=\sum_{g \in \g}A_{d(g)}, \mbox { for all } y \in X$, then the Morita context is strict.
\end{itemize}
\end{teo}

\begin{proof}
(i) Is easy to see that $(\, ,)$ is a morphism of $A^{\g_x}$-bimodules and $[\, ,]$ is a morphism of $A \#_{\alpha}^{\g} X$-bimodules. For that $[A \#_{\alpha}^{\g} X, V^x, ^xV, A^{\g_x}, [\, ,], (\, ,)]$ be a Morita context it remains to show that the conditions of associativity are valid. For any $y,z,t \in X$, such that $y \in X_{f}$, $z \in X_{f'}$ and $t \in X_{f''}$, $f,f',f'' \in \g_0$, consider $a_g \in V_{y,x}$, $b_k \in V_{x,z}$ and $c_l \in V_{t,x}$. Hence
\begin{align*}
a_g \cdot [b_k,c_l]&= a_g \cdot (b_kc_l\delta_t )\\
&= \left\{
\begin{array}{ll}
(a_gb_kc_l)_t,  \text{ if } f=d(g)=r(k)=f'\\
0,  \text{ otherwise}
\end{array}
\right. \\
&= \left\{
\begin{array}{ll}
a_gb_kc_l,  \text{ if } f=f' \text{ and } \alpha_{kl}(t)=y\\
0,  \text{ otherwise,}
\end{array}
\right.
\end{align*}
because $a_gb_kc_l \in V_yA_{kl} \subseteq V_{\alpha_{(kl)^{-1}}(y)}$. On the other hand,
\begin{align*}
(a_g,b_k)c_l &= \left\{
\begin{array}{ll}
a_gb_kc_l,  \text{ if } f=d(g)=r(k)=f' \text{ and } gk \in \g_x \\
0,  \text{ otherwise.}
\end{array}
\right.
\end{align*}
If $f=f'$, it is easy to check that $\alpha_{kl}(t)=y$ if and only if $gk \in \g_x$. Hence $a_g \cdot [b_k,c_l]=(a_g,b_k)c_l$. Similarly, if $a_g \in W_y$, $b_k \in V_z$ and $c_l \in W_t$, then $[a_g,b_k]\cdot c_l=a_g \cdot (b_k,c_l)$. Therefore, $[A \#_{\alpha}^{\g} X, V^x, {^xV}, A^{\g_x}, [\, ,], (\, ,)]$ is a Morita context.

(ii) Suppose that $X_e \neq \emptyset$, for all $e \in \g_0$, and the Morita context is strict, i.e., $[\, ,]$ is surjective. Clearly, for all $y \in X$, we have that $\sum_{{g \in \g}\atop{\alpha_g(y)=x}}A_{g^{-1}}A_g \subseteq \sum_{g \in \g}A_{d(g)}$. Conversely, take $y \in  X$ and $a_{d(h)} \in \sum_{g \in \g}A_{d(g)}$. Let $u \in X$ such that $u \in X_{d(h)}$. Then $a_{d(h)}\delta_u \in A \#_{\alpha}^{\g} X$. Hence there exists $\sum_{i=1}^n w_i \otimes v_i \in {^xV} \otimes_{A^{\g_x}} V^x$ such that
\begin{align*}
a_{d(h)}\delta_u=\sum_{i=1}^n[w_i,v_i]=\sum_{i=1}^n\sum_{z \in X}w_i{v_i}_z\delta_z,
\end{align*}
where $v_i=\sum_{z \in X}{v_i}_z$. Consequently, $a_{d(h)}=\sum_{i=1}^nw_i{v_i}_u$. If $w_i=\sum_{z \in X}w^i_z=\sum_{z \in X}\sum_{\alpha_k(x)=z}b^i_k$ and ${v_i}_{u}=\sum_{\alpha_l(y)=x}c^i_l$, then
\begin{align*}
a_{d(h)}=\sum_{i=1}^n\sum_{z \in X}w^i_z{v_i}_{u}=\sum_{i=1}^n\sum_{z \in X}\sum_{\alpha_k(x)=z}\sum_{\alpha_l(y)=x}b^i_kc^i_l.
\end{align*}
Notice that in the equality above we should have $d(h)=kl=d(kl)=d(l)$. In this case $k=l^{-1}$ and $z=\alpha_{l^{-1}}(x)=y$. Thus,
\begin{align*}
a_{d(h)}=\sum_{i=1}^n\sum_{\alpha_l(y)=x}b^i_{l^{-1}}c^i_l \in \sum_{{g \in \g}\atop{\alpha_g(y)=x}}A_{g^{-1}}A_g .
\end{align*}

(iii) Suppose that $\sum\limits_{{g \in \g}\atop{\alpha_g(y)=x}}A_{g^{-1}}A_g=\sum_{g \in \g}A_{d(g)}$ is valid for all $y \in X$. Consider $a_h\delta_y \in A \#_{\alpha}^{\g} X$, so $y \in X_{d(h)}$. Since $1_{d(h)} \in \sum_{g \in \g} A_{d(g)}$, then in particular
\begin{align*}
1_{d(h)}=\sum_{i=1}^n\sum_{{g \in \g} \atop {\alpha_g(y)=x}}b^i_{g^{-1}}c^i_g \in \sum_{{g \in \g}\atop{\alpha_g(y)=x}}A_{g^{-1}}A_g.
\end{align*}
Hence,
\begin{align*}
a_h\delta_y=a_h1_{d(h)}\delta_y=a_h \Bigg( \sum_{i=1}^n\sum_{{g \in \g} \atop {\alpha_g(y)=x}}b^i_{g^{-1}}c^i_g \Bigg)\delta_y=\sum_{i=1}^n\sum_{{g \in \g} \atop {\alpha_g(y)=x}}[a_hb^i_{g^{-1}},c^i_g],
\end{align*}
because $a_hb^i_{g^{-1}} \in W_{\alpha_{h}(y)} \subseteq W$ and $c^i_g \in V_y$, for all $g \in \g$ such that $\alpha_g(y)=x$. Therefore, $[\, ,]$ is surjective.
\end{proof}

By the Theorem \ref{morita}, we have that $A^{\g_x}$ and $A \#_{\alpha}^{\g} X$ are Morita equivalent.

\section{A final remark}

Let $(H,m,1_H,\Delta, \varepsilon,S)$ be a weak Hopf algebra and $A$ be a left $H$-module algebra. The {\it target counital subalgebra} $H_t$ is given by $\{h \in H : (\varepsilon \otimes id)(\Delta(1)(h \otimes 1)) = h\}$. The {\it{smash product algebra $A\# H$}} is defined in \cite{dmitri} as the $\ku$-vector space $A\otimes_{H_t}H$ with multiplication
\begin{align*} 
(a\# g)(b\# h)&=a(g_1\cdot b)\# g_2h,
\end{align*}
for all $a,b\in A$ and $g,h\in H$. Clearly this algebra is associative with unity $1_A \# 1_H$.

The {\it groupoid algebra} $\ku \g$ is the $\ku$-vector space with basis $\Lambda=\{g: g \in \g \}$ and multiplication induced by composition of $\g$. This is a weak Hopf algebra via $\Delta(g)= g \otimes g$, $\varepsilon(g)=1$ and $S(g)=g^{-1}$, for all $g \in \g$. If $\g$ is finite, then $\ku\g^*$, the dual of groupoid algebra, also is a weak Hopf algebra. Explicitly, if $\{v_g: g \in \g \}$ the dual basis of $\Lambda$, then 
\begin{align*}
\Delta(v_g)=\sum_{h \in D_{d(g)}} v_{gh^{-1}} \otimes v_h, & &
\varepsilon(v_g) &= \left\{\begin{array}{ll}
1,\mbox{ if } g\in \g_0\\
0,\mbox{ otherwise}
\end{array}
\right.
\mbox{and }& & S(v_g)=v_{g^{-1}},
\end{align*}
for all $g \in \g$.

The existence of a $\g$-grading on $A$ is equivalent to say that $A$ is a $\ku \g^*$-module algebra, with action given by $v_g \cdot a=a_g$, for all $a \in A$ and $g \in \g$ (cf. \cite{C}). In this case, the corresponding smash product was characterized in \cite[Proposition 1.2]{FMP}, namely
\begin{align} \label{smashdmitri}
A \# \ku\g^*=\oplus_{d(g)=r(h)}A_g \otimes v_h.
\end{align}	
	
The smash product given in \eqref{definsmash} generalizes the smash product defined in \eqref{smashdmitri}. Indeed, let $X=\{v_g : g \in \g\}$. Then $X$ is a $\g$-set via $\alpha=(X_g,\alpha_g)_{g \in \g}$, where $X_g=\{v_h: h \in R_{r(g)}\}$ and $\alpha_g: X_{g^{-1}} \to X_g$ is given by $\alpha_g(v_h)=v_{gh}$, for all $g \in \g$. Is immediate to verify that $\psi: A \# \ku\g^* \to A \#_{\alpha}^{\g} X$ given by $\psi(\sum _{d(g)=r(h)} a_g \# v_h)= \sum_{d(g)=r(h)}a_g \delta_{v_h}$ is an isomorphism of algebras. Then we have the following result, which is a consequence of Theorem \ref{dualidade}.

\begin{prop}
Let $\g$ be a finite groupoid and $A$ a $\g$-graded algebra. Then 	
\begin{align*}
(A \# \ku\g^*) *_{\gamma} \g \simeq End({A \# \ku\g^*}_{A \#_{\lambda}^{\g} O^{\g}}).
\end{align*}
	
\end{prop}
\begin{proof}
	
As we observed in the previous paragraph, $X=\{v_g : g \in \g\}$ is a $\g$-set via $\alpha=(X_g,\alpha_g)_{g \in \g}$. Moreover $X$ is a $\g$-set via $\beta = (Y_g, \beta_g)_{g \in \g}$ where $Y_g=\{v_h: h \in D_{r(g)}\}$ and $\beta_g: Y_{g^{-1}} \to Y_g$ is defined by $\beta_g(v_h)=v_{hg^{-1}}$, for all $g \in \g$ and $h \in Y_{g^{-1}}$. 
It is easy to verify that with this action $X$ is a split $(\g_{\alpha}, \g_{\beta})$-set and 
 that the action $\beta$ is fully faithful. The result now follows from Theorem \ref{dualidade}.
\end{proof}

\bibliographystyle{amsalpha}
{
}

\end{document}